\documentclass[12pt,a4paper]{article}

\usepackage[vmargin=1.1in,hmargin=1in,centering]{geometry}

\usepackage{datetime}

\usepackage[shortlabels]{enumitem}

\usepackage[T1,T5]{fontenc}

\usepackage[all]{xy}

\usepackage{amssymb, amsmath, amsthm}

\usepackage{mathptmx}
\usepackage[scaled=.90]{helvet}
\usepackage{courier}

\numberwithin{equation}{section}

\usepackage{hyperref}

\hypersetup{
  colorlinks = true,
  linkcolor = black,
  urlcolor = blue,
  citecolor = black,
}

\usepackage[
color=orange!50,
bordercolor=black,
textwidth=.8in]
{todonotes}

\renewcommand{\phi}{\varphi}

\newcommand{\bbR}{\mathbb{R}}
\newcommand{\bbZ}{\mathbb{Z}}

\newcommand{\rmE}{\mathrm{E}}
\newcommand{\rmF}{\mathrm{F}}
\newcommand{\rmH}{\mathrm{H}}

\newcommand{\rmT}{\mathrm{T}}

\newcommand{\rmd}{\mathrm{d}}

\DeclareMathOperator{\id}{id}
\DeclareMathOperator{\pr}{pr}

\DeclareMathOperator{\rank}{rank}

\DeclareMathOperator{\Tor}{Tor}

\DeclareMathOperator{\FR}{FR}
\DeclareMathOperator{\CR}{CR}
\DeclareMathOperator{\HR}{HR}

\newcommand{\ho}{/\!\!/}
\newcommand{\fin}{\mathrm{fin}}
\newcommand{\tr}{\mathrm{tr}}

\newtheorem{theorem}{Theorem}[section]
\newtheorem{theorem*}{Theorem}
\newtheorem{proposition}[theorem]{Proposition}

\newtheorem{corollary}[theorem]{Corollary}
\newtheorem{lemma}[theorem]{Lemma}

\newtheorem{letterthm}{Theorem}

\theoremstyle{definition}

\newtheorem{example}[theorem]{Example}

\newtheorem{remark}[theorem]{Remark}

\title{\bf The homology of permutation racks}

\author{Victoria Lebed \and Markus Szymik}

\newdateformat{mydate}{\monthname~\twodigit{\THEYEAR}}
\date{\mydate\today}

\begin{document}

\maketitle


\renewcommand{\abstractname}{\vspace{-2\baselineskip}}

\begin{abstract}%
\noindent
Despite a blossoming of research activity on racks and their homology for over two decades, with a record of diverse applications to central parts of contemporary mathematics, there are still very few examples of racks whose homology has been fully calculated. 
In this paper, we compute the entire integral homology of all permutation racks. 
Our method of choice involves homotopical algebra, which was brought to bear on the homology of racks only recently. 
For our main result, we establish a spectral sequence, which reduces the problem to one in Borel's equivariant homology, and for which we show that it always degenerates. 
The blueprint given in this paper demonstrates a high potential for further exploitation of these techniques. 
 
\vspace{\baselineskip}
\noindent 
20N02 
(18G40, 
18G50, 
55N91, 
55T99). 

\vspace{\baselineskip}
\noindent 
Permutations, racks, rack homology, equivariant Borel homology, spectral sequences.

\end{abstract}
Racks are fundamental algebraic structures boasting applications to knots~\cite{MR1194995,ElhamNel}, singularities~\cite{Brieskorn}, monodromy~\cite{Yetter}, branched covers~\cite{EVW,Randal-Williams}, Yang--Baxter equations~\cite{LeSurvey,LV_StrGroups,CJO_Minimality}, Hopf algebras~\cite{AG}, and the integration of Leibniz algebras~\cite{Kinyon}, to name just a few. Just as group homology is an important invariant of groups, racks come with rack homology~\cite{FRS}, and these rack homology groups are what the applications most often require. So far, complete computations of homological invariants of racks were limited to a few isolated cases~\cite{Clauwens:cohomology, Nosaka, FGG}. In this paper, we pursue the approach to rack homology via Quillen's homotopical algebra, initiated in~\cite{Szymik:Quillen}. It allows us to compute the entire integral homology for a whole family of racks: the permutation racks. 

Recall that a {\it rack}~$(X,\rhd)$ is a set~$X$ together with a binary operation~$\rhd$ such that the left multiplications~\hbox{$y\mapsto x\rhd y$} are automorphisms for all~$x$ in~$X$.~(Automorphisms are, as usual in algebra, bijective morphisms, and the left multiplications are morphisms if and only the self-distributivity equation~\hbox{$x\rhd(y\rhd z)=(x\rhd y)\rhd(x\rhd z)$} holds for all~\hbox{$x,y,z$} in~$X$.) Basic examples are the {\it permutation racks}~$(X,\phi)$, where these automorphisms are independent of the left factor: we have~\hbox{$x\rhd y=\phi(y)$} for a permutation~$\phi$ on the set~$X$. We note that permutation racks~$(X,\phi)$ are, essentially, sets~$X$ with an action of the infinite cyclic group~$\bbZ$, or~$\bbZ$--sets for short: the integer~$n$ acts as~$\phi^n$. This basic observation will play a crucial part in the paper. For instance, it suggests refining the set~$X/\phi$ of orbits to the {\it homotopy orbit space}~$X\ho\phi$ of the action, also known as the~{\it Borel construction}. For the infinite cyclic group~$\bbZ$, the resulting space is nothing but the mapping torus of the generator~$\phi$, and computations become accessible. Our main result reads as follows:

\begin{letterthm}\label{T:A}
For every permutation rack~$(X,\phi)$, there is a spectral sequence of homological type whose~$\rmE^2$ page is given by
\[
\rmE_{\bullet,q}^2\cong\overline{\rmH}_\bullet(X\ho\phi)^{\otimes(q-1)}\otimes\rmH_\bullet(X\ho\phi),
\]
and which abuts to the rack homology~$\HR_\bullet(X,\phi)$ of~$(X,\phi)$. This spectral sequence always degenerates from its~$\rmE^2$ page on.
\end{letterthm}

The spectral sequence in Theorem~\ref{T:A} is constructed in Theorem~\ref{thm:ss} of the main text, and the degeneracy is proven in Theorem~\ref{thm:degeneracy}. As for the description of the~$\rmE^2$ page, the spaces~$X\ho\phi$ are the homotopy orbit spaces of the permutation~$\phi$ acting on~$X$. We review their definition in Section~\ref{sec:Borel}. They decompose, up to homotopy, into a disjoint union of contractible lines, one for each infinite orbit of~$\phi$, and circles, one for each finite orbit of~$\phi$. In particular, we know their homology, the~$\bbZ$--equivariant Borel homology of~$(X,\phi)$, see~Proposition~\ref{prop:Borel_homology}. This extra information makes the spectral sequence efficient. The rest of this introduction spells out the computational consequences of Theorem~\ref{T:A}.

For{\it~free} permutation racks, where all orbits are infinite, there is no difference between the homotopy orbit space~$X\ho\phi$ and the usual orbit space~$X/\phi$, which is discrete. This implies that, in this particular case, the spectral sequence has only one column and degeneracy is obvious.

\begin{letterthm}
Let~$(X,\phi)$ be a free permutation rack with the set of orbits~$S=X/\phi$. There are natural isomorphisms
\[
\HR_n(X,\phi)\;\cong\;\overline{\bbZ S}^{\otimes(n-1)}\otimes\bbZ S,
\]
where~$\bbZ S$ is the free abelian group on the set~$S$, and~$\overline{\bbZ S}$ is the free abelian subgroup of linear combinations whose coefficients add up to zero. In particular, the homology is a free abelian group in each degree. If the number~$|S|=r$ of orbits is finite, then
\[
r(r-1)^{n-1}=(r-1)^n+(r-1)^{n-1}
\]
is the~$n$--th Betti number of the free permutation rack.
\end{letterthm}

Actually, we use this result in the proof of Theorem~\ref{T:A}, and therefore provide an independent proof beforehand, as Theorem~\ref{thm:free}. 

In general, the spectral sequence is concentrated in the region of the first quadrant where~\hbox{$p\leqslant q$}. This easily implies the vanishing of the differentials in low dimensions and gives:

\begin{letterthm}\label{thm:C}
For any permuation rack~$(X,\phi)$ we have
\begin{align*}
\HR_0(X,\phi)&\cong\bbZ,\\
\HR_1(X,\phi)&\;\cong\;\bbZ S\;\cong\;\bbZ^r,\\
\HR_2(X,\phi)&\;\cong\; (\overline{\bbZ S}\otimes\bbZ S)\oplus\bbZ S_\fin\;\cong\;\bbZ^{r(r-1)+r_{\fin}},
\end{align*}
where~$S_\fin\subseteq S$ is the subset of finite orbits, and~$|S_\fin|=r_\fin$ denotes its size.
\end{letterthm}

These formulas for~$\HR_0$ and~$\HR_1$ agree with the known formulas for general racks. However, our computations for~$\HR_2$, especially important for applications, are new.

While our spectral sequence is built upon our direct computation for the free permutation racks, we can use it to compute the entire integral homology for the opposite extreme as well: for permutation racks that do not contain any free orbit, such as all~\textit{finite} permutation racks.

\parbox{\linewidth}{%
\begin{letterthm}\label{thm:D}
Let~$(X,\phi)$ be a finite permutation rack or, more generally, a permutation rack such that all its orbits are finite. Then
\[
\HR_n(X,\phi)\;\cong\; (\bbZ S)^{\otimes n}
\]
is a free abelian group on the set of~$n$--tuples of orbits, where~$S=X/\phi$ as before. In particular, there is no torsion in the homology.
\end{letterthm}
}

This is proven as Theorem~\ref{thm:finite} and Remark~\ref{rem:infinite} in the main text. Briefly, we show that the functional equation satisfied by the Poincar\'e series of the~$\rmE^2$ page implies that the upper bound on the homology given by the~$\rmE^2$ page coincides with the lower bound known from the rational computations in~\cite{EtGr}: the~$n$--th Betti number of~$(X,\phi)$ equals~$r^n$, where~\hbox{$r=|X/\phi|$} is the number of orbits. In the case of a single finite orbit, the result recovers the computation for cyclic racks from~\cite[Thm.~6]{Lebed_Betti}. 

In the final Section~\ref{sec:degeneracy}, we prove the degeneracy of the spectral sequence in general, using most of our earlier results. As a consequence, in Corollary~\ref{cor:Poincare}, we can compute the entire integral homology of any permutation rack:

\begin{letterthm}
Let~$(X,\phi)$ be any permutation rack. Then the homology is a free abelian group. In particular, it is torsion free. The Poincar\'e series is given as
\[
\sum_{n=0}^\infty\rank\HR_n(X,\phi)T^n
=\frac{1+\rmT}{1-(r-1)\rmT-r_\fin\rmT^2}
\]
if the number~$r$ of orbits is finite; here~$r_\fin$ is the number of finite orbits.
\end{letterthm}

Alternatively, the~\textit{Betti numbers}~$\beta_n=\rank\HR_n(X,\phi)$ can be computed by the following recursive formula:
\[
\beta_n=
\begin{cases}
1, &\text{ if }n = 0,\\
r, &\text{ if }n = 1,\\
(r-1)\beta_{n-1} + r_\fin\beta_{n-2} &\text{ if }n\geqslant 2.
\end{cases}
\]

Beyond this definite result, we expect that the techniques of proof exposed in this paper will find applications to other computations of rack and related homologies.


\section{Preliminaries on the homology of permutation racks}

In this section, we present some general tools that help us produce and detect homology classes in permutation racks.

Let~$X$ be a set. We will write~$\bbZ X$ for the free abelian group on~$X$. Its elements are the formal linear combinations of the elements of~$X$. The subgroup~$\overline{\bbZ X}$ of~$\bbZ X$ consists of those linear combinations whose coefficients sum up to zero.

Let~$\rmT(\bbZ X)$ be the tensor algebra on the free abelian group~$\bbZ X$. This a graded abelian group that is free in each degree: in degree~$n$, we have the group~$(\bbZ X)^{\otimes n}=\bbZ(X^n)$, with basis~$X^n$. Using the algebra structure, we can write the elements as non-commutative monomials~$x_1\dots x_n$ rather than~$x_1\otimes\dots\otimes x_n$ or~$(x_1,\dots,x_n)$, which might have been more precise.

If~$(X,\rhd)$ is a rack, then the tensor algebra supports a differential of degree~$-1$ which makes it a chain complex:
\begin{align*}
\rmd(x_1\cdots x_n)
&=\sum_{k=1}^{n-1}(-1)^{k-1} 
( 
x_1\cdots x_{k-1}x_{k+1}\cdots x_{n}
-
x_1\cdots x_{k-1}(x_k\rhd x_{k+1})\cdots(x_k\rhd x_{n}) 
)\\
&=\sum_{k=1}^{n-1}(-1)^{k-1} 
x_1\cdots x_{k-1}(
x_{k+1}\cdots x_{n}
-
(x_k\rhd x_{k+1})\cdots(x_k\rhd x_{n})).
\end{align*}
This chain complex is denoted by~$\CR_\bullet(X,\rhd)$, and its homology~$\HR_\bullet(X,\rhd)$ is the~\textit{rack homology} of the rack~$(X,\rhd)$. Note that the differential is {\it not} a derivation with respect to the tensor algebra structure. 


\begin{example}\label{ex:trivial}
A rack~$(S,\rhd)$ is{\it~trivial} if the rack operation~$\rhd=\pr_2$ is given by the projection, that is~$s\rhd t=t$ for all~$s$ and~$t$ in the set~$S$. In that case, the differential on the chain complex~$\CR_\bullet(S,\pr_2)$ is zero by direct inspection, and therefore the homology of the trivial rack on~$S$ is the tensor algebra:~\hbox{$\HR_n(S,\pr_2)=(\bbZ S)^{\otimes n}$}. 
\end{example}

Trivial racks are examples of permutation racks~$(S,\pr_2)=(S,\id)$, where~$\id=\id_S$ refers to the identity permutation on~$S$. It turns out that the trivial  permutation racks help us understand all other permutation racks by means of functoriality: Let~$(X,\phi)$ be a permutation rack. The canonical projection~$X\to X/\phi$ is a morphism of racks if the set~$X/\phi=S$ of orbits is endowed with the trivial rack structure as in Example~\ref{ex:trivial}. This leads to an induced homomorphism
\begin{equation}\label{eq:test}
\HR_n(X,\phi)\longrightarrow\HR_n(X/\phi,\id)=(\bbZ S)^{\otimes n}
\end{equation}
in homology. This homomorphism gives us a means to detect elements in the homology of permutation racks. We remark that the underlying homomorphism on the chain level is obviously surjective, but this does not have to be the case on the level of homology.


If~$(X,\phi)$ is a permutation rack, the permutation~$\phi$ on~$X$ extends to an automorphism of the tensor algebra, and the formula for the differential becomes
\begin{align*}
\rmd(x_1\cdots x_n)
&=\sum_{k=1}^{n-1}(-1)^{k-1} 
x_1\cdots x_{k-1}(
x_{k+1}\cdots x_{n}
-\phi(x_{k+1}\cdots x_{n}))\nonumber\\
&=x_2\cdots x_n-\phi(x_2\cdots x_n)-x_1\rmd(x_2\cdots x_n).
\end{align*}
We can rewrite this formula as
\begin{equation}\label{eq:diff}
\rmd(xw)=w-\phi(w)-x\rmd(w)
\end{equation}
for all~$x$ in~$X$ and~$w$ in~$\CR_\bullet(X,\phi)$. Equation~\eqref{eq:diff} immediately leads to a few fundamental observations. Here is the first:

\begin{proposition}
The permutation~$\phi\colon X\to X$, which is a rack automorphism, induces the identity map~$\HR_n(X,\phi)\to\HR_n(X,\phi)$ in rack homology for all~$n$.
\end{proposition}

\begin{proof}
Recall that two morphisms~$f$ and~$g$ of chain complexes are chain homotopic if their difference is of the form~$f-g=\rmd h+h\rmd$ for an operator~$h$ that increases degrees by one. Then~$f$ and~$g$ induce the same homomorphism in homology. In the present situation, the formula~\eqref{eq:diff} for the differential shows that multiplication by any element in the rack~$X$ is such an operator for~$f=\id$ and~$g=\phi$.
\end{proof}


Here is another consequence of Equation~\eqref{eq:diff}:

\begin{lemma}\label{lem:S_first}
Let~$T$ be a subset of~$(X,\phi)$ that meets every orbit at least once. Then, modulo boundaries, every~$n$--chain in~$\CR_n(X,\phi)$ can be represented as a linear combination of basis elements~$w=x_1x_2\dots x_n$ that start in~$T$ in the sense that~$x_1\in T$.
\end{lemma}

This lemma shows that in our case it might (and will!) be fruitful to change the usual order of things and study all chains up to boundaries before restricting ourselves to cycles only.

\begin{proof}
It suffices to show that every basis element~$w=x_1x_2\dots x_n$ can be represented, modulo boundaries, as a linear combination of elements that start in~$T$. 

Let~$t\in T$ be any element. Then we have~$\rmd(tw)=w-\phi(w)-t\rmd(w)$ by the formula~\eqref{eq:diff} for the differential. This says that, modulo boundaries, we can make the difference between~$w$ and~$\phi(w)$ start with any chosen~$t\in T$. By induction, the same holds for all elements~$\phi^k(w)$ in the same orbit as~$w$. It remains to be noticed that~$\phi^k(w)$ itself starts with an element~$t\in T$ at some point: take an integer~$k$ such that~$\phi^k(x_1)\in T$. This exists by our assumption on the set~$T$.
\end{proof}

\begin{remark}\label{rem:S_first}
The lemma with its proof are valid for any rack~$R$, not just permutation racks, and any generating set~$T$ of~$R$.
\end{remark}


Recall the suspension~$C[1]$ of a chain complex~$C$: we have~$C[1]_n=C_{n-1}$ and~$\rmd_{C[1]}=-\rmd_C$. The name is justified by the equation~$\rmH_n(C[1])=\rmH_{n-1}(C)$.

\begin{proposition}\label{prop:morphism}
For all~$v\in\overline{\bbZ X}$, the multiplication with~$v$ is a morphism
\[
v\colon\CR_\bullet(X,\phi)[1]\longrightarrow\CR_\bullet(X,\phi)
\]
of chain complexes.
\end{proposition}

\begin{proof}
Let us first assume that~$v=x-y$ is the difference of two basis elements. The formula~\eqref{eq:diff} for the differential implies
\begin{equation}\label{eq:diff_diff}
\rmd((x-y)w)=-(x-y)\rmd(w)
\end{equation}
for all chains~$w$ in the complex~$\CR_\bullet(X,\phi)$, when all differentials are computed in~$\CR_\bullet(X,\phi)$. This is the equation saying that multiplication with~\hbox{$v=x-y$} is compatible with the differentials when we use the differential from~$\CR_\bullet(X,\phi)[1]$ on one side: the suspension accounts for the sign and is required because the map is multiplication
by a degree 1 class.

In general, every element~$v\in\overline{\bbZ X}$ is a linear combination of elements of the form~$x-y$, and multiplication with that element~$v$ is the corresponding linear combination of the multiplications with the~\hbox{$x-y$}'s.
\end{proof}

Proposition~\ref{prop:morphism} makes it easy for us to write down cycles for the homology of any permutation rack~$(X,\phi)$. Indeed, all elements~$x$ of~$X$ are~$1$--cycles. Therefore, for given elements~\hbox{$x_1,\dots, y_1,\dots$} in~$X$, the element
\begin{equation}\label{eq:cycles}
(x_1-y_1)(x_2-y_2)\cdots(x_{n-1}-y_{n-1})\,x_n
\end{equation}
is automatically an~$n$--cycle. Of course, many of these elements will be boundaries if not zero. Fortunately, the homomorphism~\eqref{eq:test} already provides a device that allows us to detect elements in the homology of permutation racks.

In general, not all homology classes can be described by cycles of the form~\eqref{eq:cycles}. Here is another device that produces new cycles from old:

\begin{proposition}\label{prop:fixed}
For all fixed points~$x$ in~$X$, the multiplication with~$x^2$ is a morphism
\[
x^2\colon\CR_\bullet(X,\phi)[2]\longrightarrow\CR_\bullet(X,\phi)
\]
of chain complexes.
\end{proposition}

\begin{proof}
This is a straightforward computation: from~\eqref{eq:diff} we have
\begin{align*}
\rmd(x^2w)
&=xw-\phi(xw)-x\rmd(xw)\\
&=xw-x\phi(w)-x(w-\phi(w)-x\rmd(w))\\
&=x^2\rmd(w)
\end{align*}
for all chains~$w$ in the complex~$\CR_\bullet(X,\phi)$.
\end{proof}

\begin{remark}\label{rmk:fixed}
This cycle construction can be generalized to any element~$x$ from a finite~$d$--element orbit of~$X$. We only need to replace the map~$x^2$ above with the trace
\begin{align*}
x_{\tr}\colon\CR_\bullet(X,\phi)[2] &\longrightarrow\CR_\bullet(X,\phi),\\
w &\longmapsto x\sum_{i=0}^{d-1}\phi^i(xw).
\end{align*}
The computation from the proof of Proposition~\ref{prop:fixed} can be adapted as follows:
\begin{align*}
\rmd(x\sum_{i=0}^{d-1}\phi^i(xw))
&=\sum_{i=0}^{d-1} (\phi^i(xw)-\phi^{i+1}(xw))
-x\sum_{i=0}^{d-1} (\phi^i(w)-\phi^{i+1}(w))
+x\sum_{i=0}^{d-1}\phi^i(x\rmd(w))\\
&= (xw-\phi^{d}(xw))-x (w-\phi^{d}(w))+x\sum_{i=0}^{d-1}\phi^i(x\rmd(w))\\
&= (xw-x\phi^{d}(w))-x (w-\phi^{d}(w))+x\sum_{i=0}^{d-1}\phi^i(x\rmd(w))\\
&=x\sum_{i=0}^{d-1}\phi^i(x\rmd(w)).
\end{align*}
\end{remark}


\section{The homology of free permutation racks}\label{sec:free}

In this section, we explain our computation of the homology of free permutation racks.

A permutation~$\phi$ on a set~$X$ is essentially the same structure as an action of the infinite cyclic group~$\bbZ$ on~$X$. A permutation rack~$(X,\phi)$ is called {\it free} if the corresponding action is free, that is, all orbits are infinite. In that case, the permutation rack is isomorphic to a permutation rack of the form~$(\bbZ\times S,\phi)$, where~$S$ is some set, the{\it~basis}, and~$\phi(n,s)=(n+1,s)$ is the permutation acting on the Cartesian product~\hbox{$\bbZ\times S$}. In that example, we will shorten~$(0,s)$ to~$s$ and identify the set~$(\bbZ\times S)/\phi$ of orbits with~$S$.

\begin{lemma}\label{lem:vanishing_criterion}
Given a free permutation rack~$(X,\phi)$ with basis~$S$, let~$c$ be any chain in~$\CR_\bullet(X,\phi)$ that differs from its image~$\phi(c)$ under~$\phi$ by a linear combination of monomials starting in~$S$. Then the chain~$c$ is zero.
\end{lemma}

\begin{proof}
We can write any chain~$c$ as a linear combination of monomials~$w=x_1\dots x_n$ with integral coefficients. Fix one monomial~$w$. We shall show that the coefficient in front of~$w$ is zero if~$c$ satisfies the assumption in the statement.

By freeness, there is exactly one integer~$m$ such that the monomial~$\phi^m(w)$ starts in~$S$. Either~$m$ is positive and~$\phi^n(w)$ does not start in~$S$ for all~$n\leqslant0$, or~$m$ is non-positive and~$\phi^n(w)$ does not start in~$S$ for all~$n\geqslant 1$. We shall show that the coefficient in front of~$w$ is zero in both cases.

Assume first that~$\phi^n(w)$ does not start in~$S$ for all~$n\leqslant0$. Then~$w$ does not start in~$S$ and, by assumption, it has to appear in~$\phi(c)$ with the same coefficient. It follows that~$\phi^{-1}(w)$ appears in~$c$ with the same coefficient as~$w$. Inductively, we see that all the~$\phi^n(w)$ with~$n\leqslant-1$ appear in~$c$ with the same coefficient. Since only finitely many coefficients can be non-zero, all the coefficients of these~$\phi^n(w)$ have to be zero, and this holds, in particular, for the one in front of~$w$.

Assume now that~$\phi^n(w)$ does not start in~$S$ for all~$n\geqslant 1$. Then the monomial~$\phi(w)$ appears in~$c$ and~$\phi(c)$ with the same coefficient, which is the coefficient of~$w$ in~$c$. Inductively, we see that all~$\phi^n(w)$ with~$n\geqslant0$ have the same coefficient in~$c$. As above, this coefficient has to vanish. 
\end{proof}

\begin{lemma}\label{lem:epi}
Given a free permutation rack~$(X,\phi)$ with basis~$S$, any~$n$--cycle~$c$, with~$n\geqslant2$, can be represented in~$\overline{\bbZ S}\otimes\CR_{n-1}(X,\phi)$ modulo boundaries.
\end{lemma}

\begin{proof}
We already know from Lemma~\ref{lem:S_first} that we can represent any chain~$c$  modulo boundaries by an element that lives in~$\bbZ S\otimes\CR_{n-1}(X,\phi)$. This means that we can assume~$c$ to be of the form
\[
c=\sum_{s\in S} sc_s
\]
for suitable chains~$c_s$. We choose an element~$t$ in~$S$ and rewrite the expression for~$c$ as
\[
c=\sum_{s\in S} (s-t)c_s+t\left(\sum_{s\in S}c_s\right).
\]
Since~$s-t\in\overline{\bbZ S}$, it is sufficient to show that the rightmost sum vanishes. 

We recall that~$c$ is assumed to be a cycle, so that we know 
\[
0=\rmd(c)=-\sum_{s\in S}(s-t)\rmd(c_s)+\left(\sum_{s\in S}c_s\right)-\phi\left(\sum_{s\in S}c_s\right)-t\rmd\left(\sum_{s\in S}c_s\right)
\]
from~\eqref{eq:diff_diff} and~\eqref{eq:diff}. We see that the sum~$\sum_{s\in S}c_s$ satisfies the conditions of Lemma~\ref{lem:vanishing_criterion}: it differs from its image under~$\phi$ only by monomials that start in~$S$. It follows from the lemma that the sum is zero, as claimed.
\end{proof}


We are now ready for the main result of this section.

\begin{theorem}\label{thm:free}
For any sets~$S$, let~$(\bbZ\times S,+1)$ be the free permutation rack on~$S$, with the permutation~\hbox{$+1(n,s)=(n+1,s)$}. The natural homomorphism
\[
\overline{\bbZ S}^{\otimes(n-1)}\otimes\bbZ S\overset{\cong}{\longrightarrow}\HR_n(\bbZ\times S,+1),
\]
given by multiplication, is an isomorphism.
\end{theorem}

Taking a one-element set~$S$, we recover the homology of the free monogenic rack, cf.~\cite{FRS07,FGG}: 
\[
\HR_n(\bbZ,+1)\cong
\begin{cases}
\bbZ &\text{ if } n=0,1,\\
0 &\text{ if }n\not=0,1.
\end{cases}
\]

\begin{proof}
Let us first note that the map is well-defined: Take a typical  element~\hbox{$v_1\otimes\dots\otimes v_{n-1}\otimes s$} with~\hbox{$v_j\in\overline{\bbZ S}$} and~$s\in S$. It follows from Proposition~\ref{prop:morphism} and the discussion following it that this defines a cycle in~$\CR_\bullet(\bbZ\times S,+1)$, and then it represents a homology class.

It follows from Lemma~\ref{lem:epi}, inductively, that the resulting homomorphism is surjective onto the homology. It remains to be seen that it is injective. We do this by composing it with the morphism
\[
\HR_n(\bbZ\times S,+1)\longrightarrow(\bbZ S)^{\otimes n}
\]
from~\eqref{eq:test}. Since all maps are the identity on representatives, this composition is the inclusion, hence injective. Then the first map also has to be injective, as claimed.
\end{proof}

\begin{corollary}
The homology of a free permutation rack~$(X,\phi)$ is free as an abelian group. In particular, it is torsion-free. If the number of orbits is finite, say~$r$, then the Betti numbers are given by
\[
\rank\HR_n(X,\phi)=(r-1)^{n-1}r=(r-1)^n+(r-1)^{n-1}
\]
for all~$n\geqslant 1$.
\end{corollary}

\begin{remark}
The free permutation rack is the product of racks:
\[
(\bbZ\times S,+1)\cong(\bbZ,+1)\times(S,\id).
\]
Our theorem and Example~\ref{ex:trivial} yield the entire homology of the three racks involved, and we will use the occasion to point out that a naive version of the K\"unneth theorem for rack homology is false. Indeed, given two racks~$(X,\rhd_X)$ and~$(Y,\rhd_Y)$, and their direct product~$(X\times Y,\rhd_X\times\rhd_Y)$, it is tempting to conjecture the existence of a short exact sequence
\[
0\to\bigoplus_{p+q=n}
\HR_p(X)\otimes\HR_q(Y)
\longrightarrow
\HR_n(X\times Y)
\longrightarrow\bigoplus_{p+q=n-1}\Tor^1_\bbZ(\HR_p(X),\HR_q(Y))
\to0.
\]
For a free permutation rack, this sequence would imply
\begin{align*}
\HR_n(\bbZ\times S,+1)&\;\cong\;\bigoplus_{p+q=n}\HR_p(\bbZ,+1)\otimes\HR_q(S,\id)\\
&\cong\; (\bbZ S)^{\otimes n}\oplus(\bbZ S)^{\otimes(n-1)},
\end{align*}
contradicting our theorem.
\end{remark}

\begin{remark}\label{rem:FreeRack}
Similar ideas yield a very explicit computation of the homology of the free rack~$\FR_n$ on a generating set~$\{x_1,\ldots,x_n\}$ with~$n$ elements, considerably more concise than the method from~\cite{FGG}. Recall that such a free rack~$\FR_n$ can be modelled on the set~\hbox{$\rmF_n\times\{x_1,\ldots,x_n\}$}, where~$\rmF_n$ is the free group on the set~$\{x_1,\ldots,x_n\}$, and the rack operation is given by~\hbox{$(u,x)\rhd (w,y)=(uxu^{-1}w,y)$}. This free rack has~$n$~orbits, and the elements~$X_i:=(e,x_i)$ are convenient orbit representatives, where~$e$ is the neutral element in the group. The action of these orbit representatives is particularly simple:~$X_i\rhd (w,y) = (x_iw,y)$. We will denote the inverse of this action by~\hbox{$X_i^{-1}\rhd (w,y) = (x_i^{-1}w,y)$}. Observe that the data of~$(w,y)$ and~$X_i^{\pm 1}\rhd (w,y)$ allow us to recover the element~$X_i$. We can now use this notation to show that
\[
\HR_m(\FR_n)\cong
\begin{cases}
\bbZ &\text{ if } m=0,\\
\bbZ^n &\text{ if } m=1,\\
0 &\text{ if }m\not=0,1.
\end{cases}
\]
For~$m\leqslant 1$ the result is classical. Therefore, let us take an~$m$--cycle with~$m>1$. We want to see that it is a boundary. As explained in Remark~\ref{rem:S_first}, we can add boundaries to the cycle until all monomials appearing in it start with one of the orbit representatives~$X_i$. When that is done, we can write the cycle in the form
\begin{equation}\label{eq:c}
c=\sum_k\varepsilon_k X_{j(k)} w_k,
\end{equation}
where~$\varepsilon_k =\pm 1$ is a sign, the~$X_{j(k)}\in\{X_1,\ldots,X_n\}$ are orbit representatives, and~$w_k\in\FR_n^{m-1}$. We will now show that the terms of~$c$ cancel each other in pairs, so that the whole expression~\eqref{eq:c} is zero. To start with, an analogue of formula~\eqref{eq:diff} yields
\begin{equation}\label{eq:FreeRack}
0=\rmd(c)=\sum_k\varepsilon_k w_k -\sum_k\varepsilon_kX_{j(k)}\rhd w_k -\sum_k\varepsilon_kX_{j(k)}\rmd(w_k).
\end{equation}
Here the action~$\rhd$ is extended from~$\FR_n$ to~$\FR_n^{m-1}$ diagonally. Since the sum~\eqref{eq:FreeRack} vanishes, the monomials in~\eqref{eq:FreeRack} can be partitioned into pairs of identical monomials appearing with opposite signs. Let us fix such a partition into pairs, and let us call two of its pairs~\emph{connected} if one contains the monomial~$\varepsilon_k w_k$, and the other the monomial~$-\varepsilon_kX_{j(k)}\rhd w_k$, for the same index~$k$. Since~$w_k\in\FR_n^{m-1}$ and~$m-1>0$, we have~$w_k\neq  X_{j(k)}\rhd w_k$ for all indices~$k$. Thus, a pair with at least one monomial coming from the part~$\sum_k\varepsilon_k w_k -\sum_k\varepsilon_kX_{j(k)}\rhd w_k$ of the sum~\eqref{eq:FreeRack} is connected to one or two pairs. If the cycle~$c$ is non-zero, then this part of the sum is non-empty, and there exist connected pairs. Since the sum is finite, the connected pairs form at least
\begin{itemize}[leftmargin=.15\linewidth, rightmargin=\leftmargin, nosep]
\item[(a)] a closed chain, or
\item[(b)] an open chain where the first and the last pair involve a monomial from the part~$\sum_k\varepsilon_kX_{j(k)}\rmd(w_k)$ of the sum.
\end{itemize} 
Recall that each~$X_{j(k)}$ is in fact some~$X_j=(e,x_j)$. Thus, in any of the two situations~(a) and~(b), there is a chain of connected pairs between two monomial pairs~$\pm w$ and~$\pm w'$, both in~$\FR_n^{m-1}$,  whose first~$\FR_n$ components~$(u,x)$ and~$(u',x')$ have the same~$\rmF_n$ part:~$u=u'$ and~$x=x'$ in case~(a), and~$u=u'=e$ in case~(b). By construction, the monomials from two connected pairs are each related by the action of some~$X_i$. Thus, we get a relation of the form
\[
(u',x') = X_{i(l)}^{\pm 1}\rhd\cdots\rhd X_{i(1)}^{\pm 1}\rhd (u,x) = (x_{i(l)}^{\pm 1}\cdots x_{i(1)}^{\pm 1}  u,x),
\]
where the parentheses were omitted to enhance the readability. As a result of~$u=u'$, the non-empty word~$x_{i(l)}^{\pm 1}\cdots x_{i(1)}^{\pm 1}$ represents the neutral element of the free group~$\rmF_n$. This word, therefore, contains neighbouring elements~$x_i$ and~$x_i^{-1}$, and these elements yield a chain of connected pairs~\hbox{$\pm X_i\rhd w$,~$\pm w$,~$\pm X_i\rhd w$}~(or~\hbox{$\pm w$,~$\pm X_i\rhd w$,~$\pm w$}). Then, the part~$\sum_k\varepsilon_k w_k$ of the sum~\eqref{eq:FreeRack} contains the monomials~$w$ and~$-w$~(respectively, the part~$\sum_k\varepsilon_k X_{j(k)}\rhd w_k$ of the sum contains the monomials~$ X_i\rhd w$ and~$-X_i\rhd w$) coming from the identical monomials~$X_iw$ and~$-X_iw$ of~$c$, just with opposite signs. Therefore, we can cancel this
pair, and continue until the whole expression~\eqref{eq:c} for~$c$ becomes zero.
\end{remark}


\section{Permutations and their equivariant homology}\label{sec:Borel}

Permutation racks~$(X,\phi)$ are essentially sets~$X$ with an action of the infinite cyclic group~$\bbZ$, or~$\bbZ$--sets for short. The integer~$n$ acts as~$\phi^n$. We refer to~\cite[Sec.~3]{Szymik:Center} for a detailed discussion of the relation between racks and permutations on the categorical level. In this section, we will very briefly review the equivariant homology theory for actions of a fixed group~$G$, and then we specialize it to the case of the infinite cyclic group~$\bbZ$. Proposition~\ref{prop:Borel_homology}, the computation of the equivariant homology of permutation actions, will be used in the following Section~\ref{eq:spectral_sequence}.


Let us fix a discrete group~$G$. If the group~$G$ is not trivial, then not all~$G$--sets~$X$ are free. Therefore, we need to choose free resolutions. These are~$G$--maps~\hbox{$F_\bullet\to X$} that are equivalences, where~$F$ is a simplicial~$G$--set, and the~$G$--action on~$F$ is free~(on each set of~$n$--simplices). For instance, the classifying space~$\rmE G_\bullet$ coming from the bar construction is a contractible space on which the group~$G$ acts freely; it provides functorial free simplicial resolutions
\begin{equation}\label{eq:resolution}
\rmE G_\bullet\times X=F_\bullet\to X
\end{equation}
for all~$G$--sets~$X$. The equivariant homology~$\rmH_\bullet^G(X)$ of a~$G$--set~$X$ is the homology of the orbit space~\hbox{$F_\bullet/G\simeq\rmE G_\bullet\times_G X$}. Up to homotopy, this space does not depend on the resolution~$F_\bullet$ used to compute it, and it is common to denote this{\it~homotopy orbit space} by~$X\ho G$. These constructions work more generally for~$G$--spaces, or rather simplicial~$G$--sets, and they are interesting even for the trivial~$G$--set~$X=\star$ consisting of one fixed point only: its equivariant homology is the homology of the group~$G$.

\begin{remark}
Equivariant homology in this form was initiated by Borel~\cite{Borel:seminar}. It is obvious from the description above that Borel's equivariant homology theory agrees with Quillen's general homology theory~\cite{Quillen:summary} when specialized to the algebraic theory of~$G$--sets for a fixed group~$G$. This fact is well-known and we emphasize that it plays no role in the following. We refer to~\cite{Szymik:Quillen} for Quillen's homology theory when specialized to the algebraic theory of racks.
\end{remark}


The situation is particularly transparent for the infinite cyclic group~$G=\bbZ$ of interest to us. A~$\bbZ$--action is determined by the permutation~$\phi$ corresponding to the generator. The real line~$\bbR$ gives  us a suitable intuition for the classifying space~$\rmE\bbZ_\bullet$, as it is contractible and the group~$\bbZ$ acts freely on it. However, the topological space~$\bbR$ is not a simplicial set, and we have to use a simplicial model, such as~$\rmE\bbZ_\bullet$, for it. We then need to know the equivariant homology
\[
\rmH^\bbZ_\bullet(X)=\rmH_\bullet(X\ho\phi)
\]
of any~$\bbZ$--set~$X$. For actions of the group~$\bbZ$, the homotopy orbit space~\hbox{$X\ho\phi=\rmE\bbZ_\bullet\times_\bbZ X$} is a simplicial model for the mapping torus~$\bbR\times_\bbZ X$ of the self-map~$\phi\colon X\to X$. The homology of the mapping torus sits in a long exact sequence
\[
\xymatrix@C=3em{
\cdots\ar[r]&
\rmH_\bullet(X)\ar[r]^-{\id-\phi_\bullet}&
\rmH_\bullet(X)\ar[r]&
\rmH_\bullet(X\ho\phi)\ar[r]&
\rmH_{\bullet-1}(X)\ar[r]^-{\id-\phi_\bullet}&\cdots.
}
\]

\begin{proposition}\label{prop:Borel_homology}
Let~$\phi$ be a permutation on a set~$X$. Let~$S=X/\phi$ be the set of orbits of~$\phi$, and let~\hbox{$S_\fin\subseteq S$} be the subset of finite orbits. Then the equivariant homology of~$(X,\phi)$ is given as
\[
\rmH_n(X\ho\phi)\cong
\begin{cases}
\bbZ S &\text{\upshape if } n=0,\\
\bbZ S_\fin &\text{\upshape if }  n=1,\\
0 &\text{\upshape otherwise.}
\end{cases}
\] 
\end{proposition}

\begin{proof}
Let us first assume that~$X$ is a single orbit. If~$X$ is a free orbit, then there is no need to resolve~$X$, and~$X\ho\phi\simeq X/\phi$ is a point. This proves the claim for a single free orbit. If~$X$ is a finite orbit, then the mapping torus~$X\ho\phi$ is a circle. Therefore, or by inspection of the long exact sequence, we see that the claim is true for finite orbits, too. In general, any~$G$-set is the disjoint union of its orbits, and the homology of a disjoint union is the direct sum of the homologies. This proves the claim in general.
\end{proof}

This result is the reason why finite and infinite orbits play fundamentally different roles in the  homology of permutation racks.


\section{The spectral sequence: construction and first applications}\label{eq:spectral_sequence}

Finally, we can now address the following question: Given a set~$X$ with a permutation~$\phi$, how can we compute the cohomology of the associated permutation rack~$(X,\phi)$? Our answer follows the blueprint given by Quillen in~\cite[Sec.~8]{Quillen:summary}. He computes the associative algebra homology of commutative algebras from the knowledge of the associative algebra homology of {\it free} commutative algebras~(polynomial rings). Here we adapt his methods and deduce the homology of permutation racks from the homology of{\it~free} permutation racks via a spectral sequence.

\begin{theorem}\label{thm:ss}
For every permutation rack~$(X,\phi)$, there is a spectral sequence of homological type which has its~$\rmE^2$ page given as
\[
\rmE_{\bullet,q}^2\cong\overline{\rmH}_\bullet(X\ho\phi)^{\otimes(q-1)}\otimes\rmH_\bullet(X\ho\phi)
\]
and which abuts to the rack homology of~$(X,\phi)$. 
\end{theorem}

The homology groups on the right hand side are the ones computed in Proposition~\ref{prop:Borel_homology}. We'll spell out the consequences after the proof.

\begin{proof}
We choose a simplicial resolution~$X\leftarrow F_\bullet$ by free permutations, such as~\eqref{eq:resolution} in Section~\ref{sec:Borel}. Then we apply the rack chain complex functor~$\CR_\bullet(?)$ to it. This gives a simplicial chain complex~$\CR_\bullet(F_\bullet)$. The Moore construction, which turns a simplicial abelian group into a chain complex with differential~$\partial=\sum_j(-1)^j\partial_j$, turns this into a double complex
\[
\rmE_{p,q}^0=\CR_q(F_p).
\]
This double complex comes with two spectral sequences that converge to the same target: the homology of the totalization. We will inspect these in order.

First, let us consider the spectral sequence that computes the differential in the horizontal~$p$--direction first. For a fixed~$q$, Lemma 2.3 in~\cite{Szymik:Quillen} says that the Moore complex~$\CR_q(F_\bullet)$ is a free resolution of the abelian group~$\CR_q(X)$. It follows that we get
\[
\rmE_{p,q}^1\cong
\begin{cases}
\CR_q(X) &\text{ if } p=0,\\
0 &\text{ if } p\not=0.
\end{cases}
\]
The vertical differential is the one coming from the rack complex, by naturality, so that we get
\[
\rmE_{p,q}^2\cong
\begin{cases}
\HR_q(X) &\text{ if }  p=0,\\
0 &\text{ if } p\not=0.
\end{cases}
\]
Since this is concentrated in the~$0$--th column, the spectral sequence necessarily degenerates from~$\rmE^2$ on. This shows that the target of this~(and the other!) spectral sequence is the rack homology of~$X$.

Second, let us consider the spectral sequence that computes the differential in the vertical~$q$--direction first. By now, we already know that its target is the rack homology of~$X$, but we haven't given a useful description of the initial terms, yet. The starting page is
\[
\rmE_{p,q}^0=\CR_q(F_p)
\]
as before. If we fix~$p$, and compute the vertical differential, we get
\[
\rmE_{p,q}^1=\HR_q(F_p),
\]
the rack homology of the free racks~$F_p$. We have already computed this in Section~\ref{sec:free}. If we let~$S_p$ be the set of orbits of~$F_p$, then Theorem~\ref{thm:free} gives a natural isomorphism
\[
\HR_q(F_p)\cong\overline{\bbZ S}_p^{\otimes(q-1)}\otimes\bbZ S_p.
\]
Naturality implies that, for a fixed degree~$q$, we have an isomorphism
\begin{equation}\label{eq:simplicial_abelian_groups}
\HR_q(F_\bullet)\cong\overline{\bbZ(X\ho\phi)}^{\otimes(q-1)}\otimes\bbZ(X\ho\phi)
\end{equation}
of simplicial abelian groups, where~$X\ho\phi$ is the simplicial set of orbits of~$F_\bullet$. It follows that the~$q$--th row of the~$\rmE^1$ page of the spectral sequences is the Moore complex of the right hand side of~\eqref{eq:simplicial_abelian_groups}. Up to equivalence, the Moore complex commutes with tensor products. The Moore complex for~$\bbZ(X\ho\phi)$ computes the homology of~$X\ho\phi$, and the Moore complex for~$\overline{\bbZ(X\ho\phi)}$ computes the{\it~reduced} homology of~$X\ho\phi$. Both can be read off immediately from our Proposition~\ref{prop:Borel_homology}. Since all of these homologies are free, the K\"unneth theorem implies that we have an isomorphism
\[
\rmE_{\bullet,q}^2\cong\overline{\rmH}_\bullet(X\ho\phi)^{\otimes(q-1)}\otimes\rmH_\bullet(X\ho\phi),
\]
as claimed.
\end{proof}

It follows from Proposition~\ref{prop:Borel_homology} that we have 
$\rmE_{p,q}^2=0$ whenever~$p>q$. This implies that all the differentials involving the~$\rmE_{p,q}^2$ with~$p+q\leqslant 2$ are zero. Since all of these abelian groups are also free, again by Proposition~\ref{prop:Borel_homology}, there are no extension problems, and we get
\begin{align*}
\HR_0(X,\phi)&\cong\bbZ\\
\HR_1(X,\phi)&\cong\rmH_0(X\ho\phi)\\
	&\cong\bbZ S\\
\HR_2(X,\phi)&\cong\Big(\overline{\rmH}_0(X\ho\phi)\otimes\rmH_0(X\ho\phi)\Big)\oplus\rmH_1(X\ho\phi)\\
	&\cong(\overline{\bbZ S}\otimes\bbZ S)\oplus\bbZ S_\fin,
\end{align*}
where~$S$ is the set of orbits of~$\phi$ on~$X$, and~$S_\fin$ is the subset of finite orbits. Theorem~\ref{thm:C} follows.

\begin{remark}\label{R:StrGroup}
Another important invariant of a rack~$(X,\rhd)$ is its~\emph{structure group}, given by the presentation
\[
G{(X,\rhd)} =\langle g_a,\, a\in X\,|\, g_a g_b = g_{a\rhd b} g_a,\, a,b\in X\rangle.
\]
This is the value of the left-adjoint to the forgetful functor from the category of groups to the category of racks which sends a group to its conjugation rack. It is easy to compute the structure group of a permutation rack~$(X,\phi)$: it is the free abelian group on the set of orbits of~$\phi$. (Indeed, the defining relation reads~\hbox{$g_a g_b = g_{\phi(b)} g_a$} for such a permutation rack. Choosing~$a=b$, we obtain~$g_b g_b = g_{\phi(b)} g_b$, hence~$g_b  = g_{\phi(b)}$ for all~$b$ in~$X$. Therefore, a set of orbit representatives suffices to generate the group~$G(X,\phi)$. Inserting the equation~$g_b  = g_{\phi(b)}$ back into the relation, we get~$g_a g_b = g_{b} g_a$, hence~$G(X,\phi)$ is abelian.) In particular, up to isomorphism, the group~$G(X,\phi)$ depends on the number of orbits of~$\phi$ only. According to our preceding computations of~$\HR_2(X,\phi)$, two permutation racks with the same (finite) number of orbits but different numbers of finite orbits then have isomorphic structure groups but non-isomorphic homology groups. Hence in this case homology turns out to be a stronger invariant.
\end{remark}

\begin{example}\label{ex:free_spec_seq}
If~$X$ is a free permutation rack, then~$\rmH_\bullet(X\ho\phi)$ is concentrated in degree~$0$, and so the spectral sequence is concentrated in the~$0$--th column from~$\rmE^2$ on. Then it degenerates, and we recover the result of Theorem~\ref{thm:free}. Of course, we have used that result already when setting up the spectral sequence, and this example only serves as a consistency check.
\end{example}

\section{The spectral sequence degenerates for finite racks}\label{sec:degeneracy_fin}

We now turn to a more substantial example and prove Theorem~\ref{thm:D}.

\begin{theorem}\label{thm:finite}
Let~$(X,\phi)$ be a finite permutation rack. Then the spectral sequence degenerates from its~$\rmE^2$ page on, and the rack homology~$\HR_n(X,\phi)$ is a free abelian group of rank~$r^n$, where~$r=|X/\phi|$ is the number of orbits of~$\phi$ acting on~$X$.
\end{theorem}

\begin{proof}
We inspect the~$\rmE^2$ page of the spectral sequence in Theorem~\ref{thm:ss}. We have a free abelian group in every bidegree~$(p,q)$, hence in every total degree~$n$. We will now show that the rank of the~$\rmE^2$ page in total degree~$n$ is~$r^n$.

Let~$f(\rmT)$ be the Poincar\'e series of~$\overline{\rmH}_{\bullet}(X\ho\phi)$. If~$r$ is the number of orbits, then we have
\begin{equation}\label{eq:f}
f(\rmT)=(r-1)+r\,\rmT
\end{equation}
by Proposition~\ref{prop:Borel_homology}. This satisfies the functional equation
\begin{equation}\label{eq:functional_equation}
1-f(\rmT)\rmT=(1+\rmT)(1-r\,\rmT),
\end{equation}
as is straightforward to verify from~\eqref{eq:f}. Note that we used that all orbits are finite; otherwise the Poincar\'e series~$f(\rmT)$ is a bit more complicated. From the description of the~$q$--th row of the~$\rmE^2$ page of the spectral sequence, we get that its Poincar\'e series is~$1$ for~$q=0$ and~\hbox{$f(\rmT)^q+f(\rmT)^{q-1}$} for all~\hbox{$q\geqslant 1$}, regardless of the orbit structure. Therefore, we have
\begin{align*}
\sum_{n\geqslant0}\Big(\sum_{p+q=n}\rank(\rmE^2_{p,q})\Big)\rmT^n
&=\sum_{p,q\geqslant0}\rank(\rmE^2_{p,q})\rmT^{p+q}\\
&=\sum_{q\geqslant0}\Big(\sum_{p\geqslant0}\rank(\rmE^2_{p,q})\rmT^p\Big)\rmT^q\\
&=1+\sum_{q\geqslant1}\Big(f(\rmT)^q+f(\rmT)^{q-1}\Big)\rmT^q\\
&=\sum_{q\geqslant0}f(\rmT)^q\rmT^q+\sum_{q\geqslant0}f(\rmT)^q\rmT^{q+1}\\
&=(1+\rmT)\sum_{q\geqslant0}f(\rmT)^q\rmT^q\\
&=\frac{1+\rmT}{1-f(\rmT)\rmT}.
\end{align*}
In our specific situation, we can use the functional equation~\eqref{eq:functional_equation} to see that this equals
\[
\frac1{1-r\,\rmT}=\sum_{n\geqslant0}r^n\rmT^n,
\]
as claimed.

It is known from the work~\cite[Cor.~4.3]{EtGr} of Etingof and Gra\~na that the~$n$--th Betti number of a finite permutation rack with~$r$ orbits is~$r^n$. So~$r^n$ is the rank of the~$\rmE^\infty$ page in total degree~$n$ as well. Given that all the groups on the~$\rmE^2$ page are free abelian, if the spectral sequence would not degenerate from its~$\rmE^2$ page on, the rank of the~$\rmE^\infty$ page would be strictly less than that of the~$\rmE^2$ page, contradicting what we have shown above. Therefore the spectral sequence does degenerate from its~$\rmE^2$ page on. Since~$\rmE^\infty\cong\rmE^2$ is free abelian, there are no extension problems, and we find that~$\HR_n(X,\phi)$ is the totalization of it, which we have just shown to be of rank~$r^n$ in total degree~$n$. 
\end{proof}

\begin{remark}\label{rem:infinite}
It is easy to generalize Theorem~\ref{thm:finite} slightly to cover all permutation racks~$(X,\phi)$ for permutations~$\phi$ without free orbits. These are the unions of their finite subracks, and homology is compatible with unions. It follows that the rack homology~$\HR_\bullet(X,\phi)$ of these racks is always free abelian in each degree.
\end{remark}


\section{The spectral sequence always degenerates}\label{sec:degeneracy}

This final section completes the proof of Theorem~\ref{T:A} by establishing the following result.

\begin{theorem}\label{thm:degeneracy}
For any permutation rack~$(X,\phi)$, the spectral sequence in Theorem~\ref{thm:ss} degenerates from its~$\rmE^2$ page on.
\end{theorem}

\begin{proof}
As in Remark~\ref{rem:infinite}, it is clear that we can assume that the permutation~$\phi$ on~$X$ has only{\it~finitely many orbits}: any differential involves only finitely many of them. 

Under this assumption, the~$\rmE^2$ page consists of finitely generated free abelian groups in each bidegree. Therefore, as in our proof of Theorem~\ref{thm:finite}, it is sufficient to demonstrate the degeneracy with{\it~rational coefficients}: this is enough to detect non-zero differentials between finitely generated free abelian groups. For the rest of the proof, we can and will assume that~$(X,\phi)$ is a permutation rack with finitely many orbits, and all homology will be with rational coefficients.

We construct a new permutation rack~$(\overline{X},\overline{\phi})$ from~$(X,\phi)$ as a quotient, by collapsing each finite orbit onto a single fixed point, leaving the free orbits as they were. The new permutation rack has the same number of orbits and of finite orbits, but it is{\it~semi-free}: the complement of the fixed point set is free. The quotient map~$(X,\phi)\to(\overline{X},\overline{\phi})$ is a morphism of (permutation) racks.

We claim that this morphism induces an isomorphism of spectral sequences from the~$\rmE^2$ pages on. To justify this claim, note that the~$\rmE^2$ pages are given in terms of the equivariant homology of the homotopy orbit spaces~$X\ho\phi$ and~$\overline{X}\ho\overline{\phi}$, respectively. The induced map~$X\ho\phi\to\overline{X}\ho\overline{\phi}$ between the homotopy orbit spaces is as follows: The components coming from free orbits are contractible, so any map is a homotopy equivalence. The components coming from the finite orbits are equivalent to circles, and the collapse of an orbit of length~$n\geqslant1$ to a single fixed point induces an~$n$--fold covering. Both maps yield isomorphisms on the level of rational homology. Together with naturality, this proves our claim: we have a morphism of spectral sequences that is an isomorphism from their~$\rmE^2$ pages on.

After all this d\'evissage, we are left to deal with a semi-free permutation rack~$(X,\phi)$ with finitely many orbits.
 We can compute the Poincar\'e series of the~$\rmE^2$ page as in the proof of Theorem~\ref{thm:finite}:
\begin{equation*}
\sum_{n\geqslant0}\Big(\sum_{p+q=n}\rank(\rmE^2_{p,q})\Big)\rmT^n=\frac{1+\rmT}{1-f(\rmT)\rmT}.
\end{equation*}
In contrast to Theorem~\ref{thm:finite}, we now have
\begin{equation*}
f(\rmT)=(r-1)+r_\fin\rmT,
\end{equation*}
where~$r$ is the number of orbits, and~$r_\fin$ is the number of finite orbits. Therefore, the Poincar\'e series of the~$\rmE^2$ page equals
\begin{equation*}
\frac{1+\rmT}{1-(r-1)\rmT-r_\fin\rmT^2}.
\end{equation*}
Equivalently, the Betti numbers~$\beta_n=\sum_{p+q=n}\rank(\rmE^2_{p,q})$ can be computed from the following recursion formula:
\begin{align}
\beta_0 &= 1,\qquad\qquad
\beta_1 = r,\label{E:Betti}\\
\beta_{n+2} &= (r-1)\beta_{n+1} + r_\fin\beta_n,\qquad n\geqslant 0.\label{E:Betti2}
\end{align}

We thus obtain an upper bound on the Poincar\'e series of the~$\rmE^\infty$ page, which agrees with the Poincar\'e series of the homology~$\HR_\bullet(X,\phi)$. We claim that this bound is sharp: they are equal. Once we have established that claim, the result follows because any non-zero differential would contradict the equality of the Poincar\'e series.

To prove the claim, we will use Propositions~\ref{prop:morphism} and~\ref{prop:fixed} to produce a sufficient supply of homology classes in~$\HR_\bullet(X,\phi)$. Choose a set~$Q\subseteq X$ of orbit representatives. Let~$F\varsubsetneq Q$ denote the set of fixed points. (The case~$Q=F$ was treated in Section~\ref{sec:degeneracy_fin}). Choose any element~\hbox{$q_*\in Q\setminus F$}. Define subsets~\hbox{$B_n\subseteq\CR_\bullet(X,\phi)$} inductively by
\begin{align*}
B_0 &=\{\,1\,\},\\
B_1 &=Q,\\
B_n &=\{\,(q-q_*) b_1\,|\,q\in Q\setminus\{q_*\},\,b_1\in B_{n-1}\,\}\cup\{\,(fq_*+q_*f-q_*q_*) b_2\,|\,f\in F,\,b_2\in B_{n-2}\,\}.
\end{align*}
By construction, the cardinalities~$|B_n|$ satisfy the same recursive formula as the Betti numbers~$\beta_n$ above, hence~$|B_n|=\beta_n$ for all~$n$. Moreover, since 
\begin{equation}\label{E:BnB'n}
fq_*+q_*f -q_*q_*= f\!f-(f-q_*)(f-q_*),
\end{equation}
 Propositions~\ref{prop:morphism} and~\ref{prop:fixed} allow us to move the differential~$\rmd$ through a chain~$b\in B_n$ all the way to the right, and show that~$b$ is a cycle. It remains to check that these cycles induce linearly independent homology classes in each~$\HR_n(X,\phi)$. We will actually prove slightly more: the images of these homology classes in~$\HR_n(S,\id)=(\bbZ S)^{\otimes n}$ under the map~\eqref{eq:test} are linearly independent. Identifying the orbit set~$S=X/\phi$ with the set of orbit representatives~$Q$, we reduce our task to proving linear independence in the free abelian group~$(\bbZ Q)^{\otimes n} =\bbZ (Q^n)$. 
 
Consider any order on the set~$Q$ of orbit representatives that has~$q_*$ as its minimal element and extend it anti-lexicographically to the basis~$Q^n$ of~$\bbZ (Q^n)$. By construction, any element~$b\in B_n$ is a product of some terms of the form~\hbox{$(q-q_*)$}, and some terms of the form~\hbox{$(fq_*+q_*f-q_*q_*)$}, possibly with another element of~$Q$ appended on the right. As any element~$b$ can also be written, uniquely, as a linear combination of monomials from~$Q^n$, let us consider the maximal monomial, say~\hbox{$m=m(b)$}, appearing in this linear combination; it is obtained from~$b$ by selecting the~$q$ from the terms~\hbox{$(q-q_*)$}, and the~$q_*f$ from~\hbox{$(fq_*+q_*f-q_*q_*)$}. On the other hand, per our design, the element~$b$ is uniquely reconstructible from its maximal monomial~$m$: Indeed, the chosen minimal element~$q_*$ can appear in~$m$ either on the right, or followed by some~$f\in F$. Replacing each subword of the form~$q_*f$ in~$m$ with~\hbox{$(fq_*+q_*f-q_*q_*)$}, and then each letter~$q$ that remained untouched, except for the rightmost letter of~$m$, with~\hbox{$(q-q_*)$}, we recover the element~$b$ from~$m$. This shows that the matrix of coefficients with respect to the chosen order is triangular, and the elements are linearly independent in~$\bbZ(Q^n)$, and this proves our claim.
\end{proof}

\begin{remark}
When we need an explicit description of the homology classes of semi-free permutation racks, then instead of~$B_n$ we can take a more manageable basis, defined recursively as follows:
\begin{align*}
B'_0 &=\{\,1\,\},\\
B'_1 &=Q,\\
B'_n &=\{\,(q-q_*) b_1\,|\,q\in Q\setminus\{q_*\},\,b_1\in B'_{n-1}\,\}\cup\{\,f\!f b_2\,|\,f\in F,\, b_2\in B'_{n-2}\,\}.
\end{align*}
Indeed, by construction, the sets~$B_n$ and~$B'_n$ have the same cardinalities, and by~\eqref{E:BnB'n} they have the same linear span. This basis can be extended to the case of general (non-semi-free) permutation racks: we only need to replace the elements~\hbox{$f\!f b_2$} with the more involved trace construction~$f_{\tr}(b_2)$ from Remark~\ref{rmk:fixed}. In the proof we preferred not to do so, and collapsed all finite orbits to fixed
points instead, keeping the main line of thought as non-computational as possible.
\end{remark}

\begin{corollary}\label{cor:Poincare}
Let~$(X,\phi)$ be a permutation rack with~$r_\fin$ of the~$r$ orbits finite. Then all its homology groups~$\HR_n(X,\phi)$ are free abelian, and the Poincar\'e series of the rack homology is
\[
\sum_{n=0}^\infty\rank\HR_n(X,\phi)T^n
=\frac{1+\rmT}{1-(r-1)\rmT-r_\fin\rmT^2}.
\]
\end{corollary}

The reader can easily check that this formula specializes to the ones given earlier in the cases when~$(X,\phi)$ is free ($r_\fin=0$) or all orbits are finite~($r_\fin=r$).

\begin{remark}
Standard manipulations transform the above formula into an equivalent one, which can be more suitable in practice:
\[
\sum_{n=0}^\infty\rank\HR_n(X,\phi)\rmT^n
=\sum_{n=0}^\infty(r-1)^n\left(\frac{1+\rmT}\rmT\right)\left(\displaystyle{\frac{\rmT}{1-r_\fin\rmT^2}}\right)^{n+1}.
\]
The recursive formulas~\eqref{E:Betti} and~\eqref{E:Betti2} are probably even more convenient for computations.
\end{remark}


\section*{Acknowledgements}

The authors are thankful to the referee for constructive remarks leading to a better readability of the paper.


\bibliographystyle{alpha}
\bibliography{permutations}


\vfill

\parbox{\linewidth}{%
Normandie Univ, UNICAEN, CNRS, LMNO, 14000 Caen, FRANCE\\
\href{mailto:lebed@unicaen.fr}{lebed@unicaen.fr}
}

\parbox{\linewidth}{%
Department of Mathematical Sciences, NTNU Norwegian University of Science and Technology, 7491 Trondheim, NORWAY\\
\href{mailto:markus.szymik@ntnu.no}{markus.szymik@ntnu.no}
}

\parbox{\linewidth}{School of Mathematics and Statistics, The University of Sheffield, Hicks Building, Hounsfield Road, Sheffield S3 7RH, UNITED KINGDOM\\
\href{mailto:m.szymik@sheffield.ac.uk}{m.szymik@sheffield.ac.uk}
}


\end{document}